\documentclass[a4paper]{amsart}
\usepackage{xypic}
\usepackage[only,mapsfrom]{stmaryrd}
\usepackage{graphicx}

\newtheorem{theorem}{Theorem}[section]
\newtheorem{lemma}[theorem]{Lemma}
\newtheorem{question}[theorem]{Question}
\newtheorem{corollary}[theorem]{Corollary}

\theoremstyle{definition}

\theoremstyle{remark}
\newtheorem{remark}[theorem]{Remark}

\newcommand{\C}{\mathbb{C}}
\newcommand{\N}{\mathbb{N}}
\newcommand{\R}{\mathbb{R}}

\DeclareMathOperator{\kernel}{ker}

\DeclareMathOperator{\codim}{codim}
\DeclareMathOperator{\rank}{rank}
\DeclareMathOperator{\id}{Id}
\DeclareMathOperator{\Diff}{Diff}

\begin{document}
\title[Exotic torus manifolds]{Exotic torus manifolds and equivariant smooth structures on quasitoric manifolds}

\author{Michael Wiemeler}
\address{Max-Planck-Institute for Mathematics, Vivatsgasse 7, D-53111 Bonn, Germany}
\email{wiemeler@mpim-bonn.mpg.de}
\thanks{The research was supported by SNF Grant No. PBFRP2-133466 and a grant from the MPG.}

\subjclass[2010]{57R55, 57S05, 57S15}
\keywords{torus manifolds, exotic smooth structures, fundamental group, diffeomorphism group}

\maketitle

\begin{abstract}
In 2006 Masuda and Suh asked if two compact non-singular toric varieties having isomorphic cohomology rings are homeomorphic.
In the first part of this paper we discuss this question for topological generalizations of toric varieties, so-called torus manifolds.
For example we show that there are homotopy equivalent torus manifolds which are not homeomorphic.
Moreover, we characterize those groups which appear as the fundamental groups of locally standard torus manifolds.

In the second part we give a classification of quasitoric manifolds and certain six-dimensional torus manifolds up to equivariant diffeomorphism.

In the third part we enumerate the number of conjugacy classes of tori in the diffeomorphism group of torus manifolds.
For torus manifolds of dimension greater than six there are always infinitely many conjugacy classes.
We give examples which show that this does not hold for six-dimensional torus manifolds.

\end{abstract}


\section{Introduction}
\label{sec:intro}

In 2006 Masuda and Suh \cite{1160.57032} asked the following:
\begin{question}
 Let \(X_1\), \(X_2\) be two compact non-singular toric varieties with isomorphic cohomology rings. Are \(X_1\) and \(X_2\) homeomorphic?
\end{question}

We discuss this question for topological generalizations of compact non-singular toric varieties, so-called torus manifolds.
A torus manifold is a \(2n\)-dimensional closed connected smooth orientable manifold with an effective smooth action of an \(n\)-dimensional torus \(T\) such that \(M^T\neq\emptyset\).

It is known  that the answer to the question of Masuda and Suh is ``no'' if one allows \(X_1\) and \(X_2\) to be torus manifolds \cite[Example 3.4]{choi_masuda_suh_pre}.

We show that for two torus manifolds \(M_1,M_2\) the following holds:
\begin{itemize}
\item[] \(M_1\),\(M_2\) have isomorphic cohomology
\item[\(\not\Rightarrow\)] \(M_1\),\(M_2\) are homotopy equivalent
\item[\(\not\Rightarrow\)] \(M_1\),\(M_2\) are homeomorphic
\item[\(\not\Rightarrow\)] \(M_1\),\(M_2\) are diffeomorphic. 
\end{itemize}

As a byproduct of this study we show that, for \(n\geq4\), every finitely presented group is the fundamental group of some \(2n\)-dimensional locally standard torus manifold.
For six-dimensional torus manifolds we have the following theorem:

\begin{theorem}[Theorem \ref{sec:fund-groups-torus-5}]
  A group \(G\) is the fundamental group of a six-dimensional locally standard torus manifold if and only if it is the fundamental group of a three-dimensional orientable manifold with boundary.
\end{theorem}

It should be noted that it follows from the classification results of Orlik and Raymond \cite{0287.57017} that the fundamental groups of four-dimensional torus manifolds are free products of cyclic groups. Furthermore, the fundamental groups of four-dimensional locally standard torus manifolds are free groups.

In the second part of this paper we give a classification of quasitoric manifolds and certain six-dimensional torus manifolds up to equivariant diffeomorphism.
One consequence of this classification is the following theorem:

\begin{theorem}[Corollary \ref{sec:six-dimens-torus-2}]
  Simply connected six-dimensional torus manifolds with vanishing odd degree cohomology are equivariantly homeomorphic if and only if they are equivariantly diffeomorphic.
\end{theorem}

In the third part of this paper we enumerate conjugacy classes of \(n\)-dimensional tori in the diffeomorphism group of \(2n\)-dimensional torus manifolds.
If \(n\geq4\), then there are always infinitely many conjugacy classes.

The enumeration in the case of four-dimensional simply connected torus manifolds was done by Melvin \cite{0486.57016}.
He showed that there are four-dimensional torus manifolds with only a finite number of conjugacy classes of two-dim\-en\-sio\-nal tori and others with an infinite number.

We also give an enumeration for some six-dimensional torus manifolds.
We show that there are six-dimensional torus manifolds with only a finite number of conjugacy classes and others with an infinite number.
This last result is based on a generalization of a theorem from \cite{wiemeler-remarks} to simply connected six-dimensional torus manifolds whose cohomology is generated in degree two (see Section~\ref{sec:class} for details).

In this article all (co-)homology groups are taken with integer coefficients.

This paper is organized as follows.
In section \ref{sec:fundamental} we describe the fundamental group of torus manifolds.
In section \ref{sec:homotopy_not_homeo} we construct homotopy equivalent but not homeomorphic torus manifolds.
In section \ref{sec:homeo_not_diffeo} we construct homeomorphic but not diffeomorphic torus manifolds.

In section \ref{sec:app-intro} we give a classification of quasitoric manifolds up to equivariant diffeomorphism.
In section \ref{sec:class} we give a classification of six-dimensional torus manifolds with vanishing odd dimensional cohomology.

In section \ref{sec:tori_in_diff} we prove that there are infinitely many non-conjugated tori in the diffeomorphism group of torus manifolds of dimension greater than six.
In section \ref{sec:tori_diff_six} we use the classification from section \ref{sec:class}  to determine the number of conjugacy classes of three-dimensional tori in the diffeomorphism groups of some six-dimensional torus manifolds.

  I would like to thank Dieter Kotschick for drawing my attention to the question which groups are isomorphic to fundamental groups of torus manifolds.
This paper grow out of an attempt to answer this question.
Moreover, I want to thank Stephen Miller and Nigel Ray for discussions on smooth structures on polytopes and quasitoric manifolds.
I would also like to thank Anand Dessai for comments on an earlier version of this paper.

\section{Fundamental groups of torus manifolds}
\label{sec:fundamental}

In this section we construct torus manifolds with non-trivial fundamental group.
Before we do so we introduce the main construction for our counterexamples to the rigidity problem for torus manifolds.
Let \(M\) be a \(2n\)-dimensional torus manifold and \(X\) a closed smooth connected oriented \(n\)-dimensional manifold.

Then let \(\iota_1:D^n\rightarrow X\) be an orientation-reversing embedding of an \(n\)-dimensional disc into \(X\).
Moreover, let \(\iota_2:T^n\times D^n\rightarrow M\) be an orientation-preserving embedding of an equivariant tubular neighbourhood of a principal orbit in \(M\).
Here \(T^n\times D^n\) is oriented in the natural way.
Then we define
\begin{equation*}
  \alpha(M,X)= T^n\times (X-\iota_1(\mathring{D}^n)) \cup_{f}(M -\iota_2(T^n\times \mathring{D}^n)),
\end{equation*}
where \(f=(\id_{T^n}\times \iota_1|_{S^{n-1}})\circ \iota_2|_{T^n\times S^{n-1}}^{-1}\).

It is immediate from the definition that \(\alpha(M,X)\) is a torus manifold of dimension \(2n\).

\begin{lemma}
\label{sec:fund-groups-torus-4}
  The equivariant diffeomorpism type of \(\alpha(M,X)\) does not depend on the choices of \(\iota_1\) and \(\iota_2\).
\end{lemma}
\begin{proof}
  Let \(\iota_1':D^n\rightarrow X\) be another orientation-reversing  embedding of the \(n\)-dimensional disc into \(X\).
Then it follows from Corollary 3.6 of \cite[p. 52]{0767.57001} that there is a diffeomorphism \(f:X\rightarrow X\) such that \(f\circ \iota_1=\iota_1'\).
Therefore the equivariant diffeomorphism type of \(\alpha(M,X)\) does not depend on \(\iota_1\).

Before we prove that the equivariant diffeomorphism type of \(\alpha(M,X)\) is independent of \(\iota_2\), we show that \(M-\bigcup_{\{e\}\neq G \subset T} M^G\) is connected. Here the union is taken over all non-trivial subgroups \(G\) of \(T\).

This is the case if, for all \(\{e\}\neq G\subset T\), \(M^G\) has at least codimension two.
Assume that there is a component of \(M^G\) with codimension one.
Then \(G\) is isomorphic to \(\mathbb{Z}/2\mathbb{Z}\).
Moreover, the orbit of a generic point in this component is equivariantly diffeomorphic to \(T/G\).
Since \(T/G\) is orientable, it follows from Proposition 3.11 of \cite[p. 185]{0246.57017} that \(\codim M^G>1\).
Therefore \(M-\bigcup_{\{e\}\neq G \subset T} M^G\) is connected.

Now assume that \(\iota_2':T^n\times D^n\rightarrow\) is another equivariant tubular neighbourhood of a principal orbit in \(M\).
Then it follows from Corollary 2.4 of \cite[p. 47]{0449.57009} and because equivariant tubular neighbourhoods of a closed invariant submanifold of a \(T\)-manifold are  equivariantly isotopic, that there is an equivariant diffeomorphism \(f:M\rightarrow M\) such that \(f\circ \iota_2=\iota_2'\).
Therefore the equivariant diffeomorphism type of \(\alpha(M,X)\) is independent of \(\iota_2\). 
\end{proof}

It follows easily from Lemma~\ref{sec:fund-groups-torus-4} that \(\alpha(M,S^n)\) is equivariantly diffeomorphic to \(M\).

A torus manifold \(M\) is called locally standard if each orbit in \(M\) has an invariant open neighborhood which is weakly equivariantly diffeomorphic to an open invariant subset of an \(n\)-dimensional complex linear faithful representation of \(T\).
The isotropy groups of points in a locally standard torus manifold are connected.
Moreover, the orbit space of a locally standard torus manifold is naturally a manifold with corners.

It follows immediately from the definition that \(\alpha(M,X)\) is locally standard if and only if \(M\) is locally standard. 

Now we prove the following theorems.

\begin{theorem}
\label{sec:fund-groups-torus}
  Let \(G\) be a group which is the fundamental group of a smooth orientable \(n\)-dimensional manifold, \(n\geq 3\). Then \(G\) is the fundamental group of a \(2n\)-dimensional locally standard torus manifold.
\end{theorem}

Since, for \(n\geq 4\), every finite presentable group is the fundamental group of some orientable \(n\)-dimensional manifold we get:

\begin{corollary}
\label{sec:fund-groups-torus-1}
  For \(2n\geq 8\), every finite presentable group is the fundamental group of some \(2n\)-dimensional locally standard torus manifold.
\end{corollary}

\begin{theorem}
\label{sec:fund-groups-torus-2}
  Let \(M\) be a torus manifold of dimension \(2n\), \(n\geq 3\), and \(X\) a homology sphere of dimension \(n\). 
  Then \(\pi_1(\alpha(M,X))=\pi_1(X)*\pi_1(M)\) and there is an equivariant map \(f:\alpha(M,X) \rightarrow M\) which induces an isomorphism in homology. 
\end{theorem}

\begin{remark}
  It follows from Theorem~\ref{sec:fund-groups-torus-2} that there is no algorithm which can decide if two torus manifolds of dimension \(2n \geq 10\) are homotopy equivalent \cite[p. 169-171]{0311.57001}.
\end{remark}

\begin{proof}[Proof of Theorem 2.2]
  Let \(X\) be a smooth orientable \(n\)-manifold, \(n\geq 3\), with fundamental group \(\pi_1(X)=G\) and \(M\) a simply connected torus manifold of dimension \(2n\).
Moreover, let \(\iota_1,\iota_2\) as in the definition of \(\alpha(M,X)\).

Since the maximal orbits in \(M\) have codimension greater than two, we have
\begin{equation*}
  \pi_1(M-\iota_2(T^n\times \mathring{D}^n))=\pi_1(M)=0.
\end{equation*}
Similarly one sees
\begin{equation*}
  \pi_1(X-\iota_1(\mathring{D}^n))=\pi_1(X)=G.
\end{equation*}
Then we have by Seifert-van Kampen's theorem
\begin{equation*}
  \pi_1(\alpha(M,X))=(\mathbb{Z}^n\times G) *_{\mathbb{Z}^n}\pi_1(M)=G.
\end{equation*}
If we take \(M\) in the above construction to be locally standard, then \(\alpha(M,X)\) is locally standard.
Therefore the statement follows.
\end{proof}

For the proof of Theorem~\ref{sec:fund-groups-torus-2} we need the following lemma.

\begin{lemma}
\label{sec:fund-groups-torus-3}
  Let \(X,X'\) be two oriented \(n\)-dimensional manifolds and \(M\) a \(2n\)-dimensional torus manifold.
  If there is an orientation preserving map \(f:X\rightarrow X'\) which induces an isomorphism in homology, then there is an equivariant map \(F:\alpha(M,X)\rightarrow \alpha(M,X')\) which induces an isomorphism in homology. 
\end{lemma}
\begin{proof}
  Since \(f\) has degree one, we may assume by a result of Hopf \cite[Theorem 4.1, p. 376]{0148.43103} that there is a disc \(\iota_1:D^n\hookrightarrow X\) such that \(f\circ \iota_1\) is an embedding and \(X-\iota_1(D^n)\) is mapped by \(f\) to \(X'-f(\iota_1(D^n))\).

By Lemma~\ref{sec:fund-groups-torus-4}, we may use the embedding \(f\circ \iota_1\) to construct \(\alpha(M,X')\).
Therefore \(\id\times f|_{T^n\times(X-\iota_1(D^n))}\) extends to a map \(F:\alpha(M,X)\rightarrow \alpha(M,X')\) which is the identity on \(M-\iota_2(T^n\times \mathring{D}^n)\). 
Let \(A=M-\iota_2(T^n\times \mathring{D}^n)\), \(M'=\alpha(M,X)\) and \(M''=\alpha(M,X')\).
 Then we have the following commutative diagram with exact rows.

\begin{equation*}
  \xymatrix{
    H_{i+1}(M',A)\ar[r]\ar[d]^{F_*}&H_i(A)\ar[d]^{\id}\ar[r]&H_i(M')\ar[r]\ar[d]^{F_*}&H_i(M',A)\ar[d]^{F_*}\ar[r]&H_{i-1}(A)\ar[d]^{\id}\\
    H_{i+1}(M'',A)\ar[r]&H_i(A)\ar[r]&H_i(M'')\ar[r]&H_i(M',A)\ar[r]&H_{i-1}(A)
}
\end{equation*}

Therefore, by the five lemma, \(F_*:H_*(M')\rightarrow H_*(M'')\) is an isomorphism if and only if \(F_*:H_*(M',A)\rightarrow H_*(M'',A)\) is an isomorphism.

By excision we have the following commutative diagram with all horizontal maps isomorphisms
\begin{equation*}
\xymatrix{
  H_i(T^n\times X, T^n\times D^n)\ar[d]^{(\id\times f)_*}&H_i(T^n\times \tilde{X},T^n\times S^{n-1})\ar[l]\ar[r]\ar[d]^{(\id\times f)_*}&H_i(M',A)\ar[d]^{F_*}\\
  H_i(T^n\times X', T^n\times D^n)&H_i(T^n\times \tilde{X}' ,T^n\times S^{n-1})\ar[l]\ar[r]&H_i(M'',A)
}
\end{equation*}
where \(\tilde{X}=X-\iota_1(\mathring{D}^n)\) and \(\tilde{X}'=X'-f(\iota_1(\mathring{D}^n))\).

Therefore \(F_*:H_i(\alpha(M,X),A)\rightarrow H_i(\alpha(M,X'),A)\) is an isomorphism if and only if \((\id\times f)_*:  H_i(T^n\times X, T^n\times D^n)\rightarrow  H_i(T^n\times X', T^n\times D^n)\) is an isomorphism.
But this follows from the exact homology sequence of the pair \((T^n\times X, T^n\times D^n)\) and the five-lemma.

This proves that \(F\) induces an isomorphism in homology.
\end{proof}

\begin{proof}[Proof of Theorem 2.4]
  There is an orientation preserving map \(X\rightarrow S^n\) which induces an isomorphism in homology.
  Therefore, by Lemma~\ref{sec:fund-groups-torus-3}, there is an equivariant map
  \(F:\alpha(M,X)\rightarrow \alpha(M,S^n)=M\) which induces an isomorphism in homology.

  Because there are \(T\)-fixed points in \(M\), the inclusion of a principal orbit in \(M\) is null-homotopic.
  Hence, the statement about the fundamental group follows as in the proof of Theorem~\ref{sec:fund-groups-torus}.
\end{proof}

As a step towards a description of those groups which are fundamental groups of six-dimensional torus manifolds, we discuss the fundamental groups of locally standard torus manifolds.

\begin{lemma}
\label{sec:fund-groups-torus-6}
  Let \(M\) be a locally standard torus manifold with orbit space \(X=M/T\).
  Then the orbit map \(M\rightarrow X\) induces an isomorphism \(\pi_1(M)\rightarrow \pi_1(X)\).
\end{lemma}
\begin{proof}
  Let \(M_0\) be the union of all principal orbits in \(M\).
  Then \(M_0\) is a principal \(T\)-bundle over \(\mathring{X}=X-\partial X\).
Therefore there is an exact sequence
\begin{equation*}
  \pi_1(T)\rightarrow \pi_1(M_0)\rightarrow \pi_1(\mathring{X})\rightarrow 1.
\end{equation*}
Because there are fixed points in \(M\), the inclusion of a principal orbit into \(M\) is null-homotopic, i.e. the composition
\begin{equation*}
  \pi_1(T)\rightarrow \pi_1(M_0)\rightarrow \pi_1(M)
\end{equation*}
is the trivial homomorphism.

Hence, there is a map \(\pi_1(\mathring{X})\rightarrow \pi_1(M)\) such that the following diagram commutes
\begin{equation*}
  \xymatrix{
    \pi_1(M_0)\ar[r]\ar[d]&\pi_1(M)\ar[d]\\
    \pi_1(\mathring{X})\ar[r]\ar[ur]&\pi_1(X)
}
\end{equation*}

Since \(\mathring{X}\) is homotopy equivalent to \(X\), it remains to show that \(\pi_1(M_0)\rightarrow \pi_1(M)\) is surjective.
We have \(M_0=M-\bigcup_{\{e\}\neq G\subset T}M^G\), where the union is taken over all non-trivial subtori of \(T\).
Because each \(M^G\) has at least codimension two in \(M\) it follows that \(\pi_1(M_0)\rightarrow \pi_1(M)\) is surjective.
\end{proof}

\begin{theorem}
  \label{sec:fund-groups-torus-5}
  A group \(G\) is the fundamental group of a six-dimensional locally standard torus manifold if and only if it is the fundamental group of a three-dimensional orientable manifold with boundary.
\end{theorem}
\begin{proof}
  By Lemma~\ref{sec:fund-groups-torus-6}, the fundamental group of a six-dimensional locally standard torus manifold is isomorphic to the fundamental group of its orbit space.
This orbit space is after smoothing the corners a three-dimensional manifold with boundary.
Since the torus \(T\) is orientable, an orientation on the torus manifold induces an orientation on the orbit space.

Now let \(X\) be a three-dimensional orientable manifold with boundary. 
Then by J\"anich's Klassifikationssatz \cite{0153.53703}, there is a five-dimensional special \(T^2\)-manifold \(Y\) with orbit space \(X\).
Since \(T^2\) is orientable, an orientation on \(X\) induces an orientation on \(Y\).
It follows from J\"anich's construction that \(Y'=Y\times S^1\) is a locally standard six-dimensional \(T^3\)-manifold without fixed points.
The orbit space of the \(T^3\)-action on \(Y'\) is \(X\).

Now let \(\iota_1:T^3\times D^3\hookrightarrow Y'\) and \(\iota_2:T^3\times D^3\hookrightarrow S^6\) be inclusions of equivariant tubular neighborhoods of principal orbits in \(Y'\) and \(S^6\), respectively.
Define
\begin{equation*}
  M=Y'-\iota_1(T^3\times \mathring{D}^3) \cup_{T^3\times S^2} S^6-\iota_2(T^3\times \mathring{D}^3).
\end{equation*}

Then \(M\) is a locally standard torus manifold. 
The orbit space \(M/T\) is diffeomorphic to a connected sum at interior points of \(X\) and \(S^6/T\).
Since \(S^6/T\) is after smoothing the corners a three-dimensional disc, it follows that \(M/T\) is homeomorphic to \(X\) with a three-dimensional open disc removed from its interior.
Therefore we have by Lemma~\ref{sec:fund-groups-torus-6} \(\pi_1(M)=\pi_1(M/T)=\pi_1(X)\).
\end{proof}

\section{Homotopy equivalent but not homeomorphic torus manifolds}
\label{sec:homotopy_not_homeo}

In this section we prove the following theorem.

\begin{theorem}
\label{sec:homot-equiv-but}
  There are simply connected torus manifolds \(M_1,M_2\) and an equivariant map \(M_1\rightarrow M_2\) which is a homotopy equivalence such that \(M_1\) and \(M_2\) are not homeomorphic.
\end{theorem}
\begin{proof}
  Let \(X\) be a homotopy \(\C P^n\), \(n\geq 3\), such that \(X\) has non-standard first Pontrjagin-class \(p_1(X)=ax^2\neq \pm (n+1)x^2\), where \(a\in \mathbb{Z}\) and \(x\in H^2(X)\) is a generator.
As in the proof of Theorem \ref{sec:fund-groups-torus} one sees that \(\alpha(S^{4n},X)\) and \(\alpha(S^{4n},\C P^n)\) are simply connected.
By Lemma~\ref{sec:fund-groups-torus-3}, there is an equivariant map
  \begin{equation*}
    M_1=\alpha(S^{4n},X)\rightarrow \alpha(S^{4n},\C P^n)=M_2
  \end{equation*}
which is a homotopy equivalence.

We claim that \(H^4(M_1)\) is torsion-free but \(p_1(M_1)\neq p_1(M_2)\).
From this claim it follows that \(M_1\) and \(M_2\) are not homeomorphic (invariance of rational Pontrjagin-classes).

We have the following exact sequence
  \begin{gather*}
  \xymatrix{
    H^4(T^{2n}\times S^{2n-1})&H^4(T^{2n}\times (X-\mathring{D}^{2n}))\oplus H^4(S^{4n}-(T^{2n}\times \mathring{D}^{2n}))\ar[l]}\\
\xymatrix{
&H^4(M_1)\ar[l]
  &H^3(T^{2n}\times S^{2n-1})\ar[l]&H^3(T^{2n}\times(X-\mathring{D}^{2n}))\oplus\dots\ar[l]_(.57)\phi
}    
  \end{gather*}
Since \(\phi\) is surjective and \(H^4(S^{4n}-(T^{2n}\times \mathring{D}^{2n}))=0\),
it follows that \(H^4(M_1)\) injects into \(H^4(T^{2n}\times (X-D^{2n}))\).

Because \(H^4(T^{2n}\times S^{2n-1})\) and \(H^4(T^{2n}\times (X-\mathring{D}^{2n}))\) are torsion-free, it follows that
\(H^4(M_1)\) is torsion-free and \(p_1(M_1)=ax'\), where \(x'\in H^4(M_1)\) is a primitive vector.
Similarly, one shows that \(p_1(M_2)=(n+1)x''\), where \(x''\in H^4(M_2)\) is a primitive vector.
Therefore \(p_1(M_1)\neq p_1(M_2)\) follows.
\end{proof}

\section{Homeomorphic but not diffeomorphic torus manifolds}
\label{sec:homeo_not_diffeo}

In this section we prove the following theorem.

\begin{theorem}
\label{sec:homeomorphic-but-not}
  There are \(18\)-dimensional torus manifolds \(M_1\), \(M_2\) such that \(M_1\), \(M_2\) are equivariantly homeomorphic but not diffeomorphic.
\end{theorem}

For the proof of this theorem we need the following lemma.

\begin{lemma}
\label{sec:homeomorphic-but-not-1}
  Let \(\iota_1,\iota_2: D_1^n\rightarrow M^n\), \(n\geq 6\), be locally flat orientation-preserving embeddings in a topological manifold \(M\).
  Then there is an orientation-preserving homeomorphism \(f:M\rightarrow M\) such that \(f\circ \iota_2|_{D^n_{1/2}}=\iota_1|_{D^n_{1/2}}\). Here \(D^n_a\) denotes the \(n\)-dimensional disc with radius \(a\).
\end{lemma}
\begin{proof}
  Since the group of orientation-preserving homeomorphisms acts transitively on \(M\) we may assume that \(\iota_1(0)=\iota_2(0)\).

 By Theorem 3.6 of \cite[p. 95-96]{0264.57004}, there is an orientation-preserving homeomorphism \(f_0:M\rightarrow M\) such that \(f_0\circ\iota_2(D_1^n)=\iota_1(D_1^n)\).
Therefore we may assume that \(\iota_1(D_1^n)=\iota_2(D_1^n)\).

By the remark after Theorem 1 of \cite[p. 575-576]{0176.22004}, the homeomorphism \(\iota_1^{-1}\circ \iota_2|_{\mathring{D}_1^n}\) of \(\mathring{D}_1^n\) is isotopic to a homeomorphism \(f_1:\mathring{D}_1^n\rightarrow \mathring{D}_1^n\) which is the identity on \(D^n_{1/2}\).

By Corollary 1.2 of \cite[p. 63]{0214.50303}, \(\iota_1\circ f_1\circ \iota_2^{-1}|_{\iota_2(D^n_{1/2})}\) extends to a homeomorphism \(f:M\rightarrow M\).
This \(f\) has the required properties.
\end{proof}

\begin{proof}[Proof of Theorem 4.1]
  There are \(9\)-dimensional orientable manifolds \(X_1\), \(X_2\) such that \(X_1\) and \(X_2\) are homeomorphic but \(p_2(X_1)=0\) and \(p_2(X_2)\neq 0\) \cite[Proof of Theorem 3.1, p. 362-363]{1212.57009}.

  From Lemma \ref{sec:homeomorphic-but-not-1} it follows as in the proof of Lemma~\ref{sec:fund-groups-torus-4} that \(M_1=\alpha(S^{18},X_1)\) and \(M_2=\alpha(S^{18},X_2)\) are homeomorphic.

  As in the proof of Theorem~\ref{sec:homot-equiv-but} one sees that \(H^8(M_1)\) injects into \(H^8(T^{9}\times (X_1- \mathring{D}^{9}))\oplus H^8(S^{18}-(T^9\times \mathring{D}^9))\).
Since \(p_2(S^{18})=0\) and \(H^8(T^{9}\times (X_1- \mathring{D}^{9}))\cong H^8(T^{9}\times X_1)\), it follows that \(p_2(M_1)=0\) and \(p_2(M_2)\neq 0\).
  Hence, \(M_1\) and \(M_2\) are not diffeomorphic.
\end{proof}

\section{Classification of quasitoric manifolds up to equivariant diffeomorphism}
\label{sec:app-intro}

The purpose of this section is to prove some facts about differentiable structures on polytopes and equivariant differentiable structures on quasitoric manifolds.

We use the following notation a (weakly) quasitoric manifold is a locally standard torus manifold such that the orbit space is face-preserving homeomorphic to a simple convex polytope.
We say that a quasitoric manifold \(M\) is strongly quasitoric if the orbit space (with the smooth structure induced by the smooth structure of \(M\)) is diffeomorphic to a simple convex polytope \(P\) (with the natural smooth structure coming from the embedding \(P\hookrightarrow \R^n\)).

We will show that this smooth structure of \(P\) depends only on the combinatorial type of \(P\).
Moreover, we will see that the smooth structures (up to equivariant diffeomorphism) on \(M\) are one-to-one to the smooth structures on \(P\).
Therefore in the equivariant homeomorphism class of a quasitoric manifold \(M\) there is up to equivariant diffeomorphism exactly one strongly quasitoric manifold.

We begin by proving that two combinatorially equivalent simple polytopes equip\-ped with their natural smooth structures are diffeomorphic.
To show that we will work with more general objects than simple polytopes so called nice manifolds with corners.

An \(n\)-dimensional manifold with corners \(X\) is locally modelled on charts \(\psi_U: X\supset U\rightarrow [0,1[^n\), where \(U\) is an open subset of \(X\).
For each \(x\in U\) we define \(c(x)\) to be the number of components of \(\psi_U(x)\) which are zero.
This number \(c(x)\) is independent of the chart \(\psi_U\) and the neighborhood \(U\) of \(x\).
So that we get a well defined function \(c:X\rightarrow \{0,\dots,n\}\).

We call the closures of the connected components of \(c^{-1}(k)\) the codimension \(k\) faces of \(X\).
The faces of codimension one are also called facets of \(X\).

A manifold \(X\) with corners is called nice if each codimension \(k\) face of \(X\) is contained in exactly \(k\) facets.
Nice manifolds with corners are also called manifolds with faces.
If \(X\) is a nice manifold with corners, then all its faces are again nice manifolds with corners.

Associated to a nice manifold with corners \(X\) is its face-poset \(\mathcal{P}(X)\) which consists out of the faces of \(X\). 
The partial ordering on \(\mathcal{P}(X)\) is given by inclusion.
We call two nice manifolds with corners of the same dimension combinatorially equivalent if they have isomorphic face-posets.

\begin{theorem}
\label{sec:class-quas-manif}
  Let \(X\) be a \(n\)-dimensional nice manifold with corners such that for \(1\leq i\leq n\) all \(i\)-dimensional faces of \(X\) are (after smoothing the corners) diffeomorphic to the \(i\)-dimensional disc \(D^i\).
Assume that there is an \(n\)-dimensional simple polytope \(P\) and
\begin{enumerate}
\item an isomorphism \(\phi:\mathcal{P}(X)\rightarrow \mathcal{P}(P)\),
\item for a line shelling (see \cite[Chapter 8]{0823.52002}) \(F_1,\dots,F_m\) of \(P\) and some \(k<m\), there is a diffeomorphism \(f\) of a neighborhood of \(\phi^{-1}(F_1)\cup\dots\cup\phi^{-1}(F_k)\) onto a neighborhood of \(F_1\cup\dots\cup F_k\) such that for each face \(F\) of \(X\) which is contained in \(\phi^{-1}(F_1)\cup\dots\cup\phi^{-1}(F_k)\) we have \(\phi(F)=f(F)\).
\end{enumerate}
Then there is a diffeomorphism \(g:X\rightarrow P\) such that for each face \(F\) of \(X\) we have \(\phi(F)=g(F)\).
Moreover, for some neighborhood \(W\) of \(\phi^{-1}(F_1)\cup\dots\cup\phi^{-1}(F_k)\) we have \(g|_W=f|_W\).
\end{theorem}
\begin{proof}
  We prove this theorem by induction on the dimension of \(X\).
  If \(X\) is one- or zero-dimensional then there is nothing to prove.
  Therefore assume that \(X\) is at least two-dimensional.

  As a first step we construct a diffeomorphism \(f':\phi^{-1}(F_{k+1})\rightarrow F_{k+1}\) which is equal to \(f\) near \(\phi^{-1}(F_1)\cup\dots\cup\phi^{-1}(F_k)\) and such that the map induced by \(f'\) on the face-poset of \(\phi^{-1}(F_{k+1})\) is equal to \(\phi|_{\mathcal{P}(\phi^{-1}(F_{k+1}))}\).
  At first assume that \(k+1<m\) then the existence of \(f'\) follows from the induction hypothesis.

  Therefore assume that \(k+1=m\).
  Then we may choose neighborhoods \(W\), \(W'\) in \(\phi^{-1}(F_{k+1})\) and \(F_{k+1}\)  of the boundary of \(\phi^{-1}(F_{k+1})\) and \(F_{k+1}\), respectively, such that \(W\), \(W'\) are nice manifolds with corners and
  \begin{align*}
    \phi^{-1}(F_{k+1})&=W\cup_h D^{n-1}& F_{k+1}&=W'\cup_{h'}D^{n-1},
  \end{align*}
  and \(f\) maps \(W\) diffeomorphically onto \(W'\).
  Here \(h\) and \(h'\) are diffeomorphisms of the \(n-2\)-dimensional sphere onto some component of the boundary of \(W\), \(W'\), respectively.

Therefore \(f\) extends to a diffeomorhism \(f':\phi^{-1}(F_{k+1})\rightarrow F_{k+1}\) if \(h'^{-1}\circ f\circ h :S^{n-2}\rightarrow S^{n-2}\) extends to a diffeomorphism of \(D^{n-1}\).
Now after smoothing the corners \(V=W\cup \bigcup_{i=1}^k \phi^{-1}(F_i)\) is diffeomorphic to \(D^{n-1}\).
Moreover, \(V\cup_{h}D^{n-1}\)
and \(V\cup_{f^{-1}\circ h'} D^{n-1}\) are diffeomorphic to the boundaries of \(X\) and \(P\), respectively.
Because these boundaries are after smoothing the corners diffeomorphic to \(S^{n-1}\), if follows that \(h'^{-1}\circ f\circ h\) extends to a diffeomorphism of \(D^{n-1}\).

As our second step in the proof we extend \(f\cup f'\) to a diffeomorphism of a neighborhood of \(\phi^{-1}(F_1)\cup\dots\cup \phi^{-1}(F_{k+1})\).
At first choose collars \(C_i\)  of \(\phi^{-1}(F_i)\), \(i=1,\dots,k\), in \(X\).
Then by \(f(C_i)\) there is given a collar of \(F_i\) in \(P\).

Let \(W=\bigcup_{i=1}^k C_i\) and \(W'=\bigcup_{i=1}^kf(C_i)\) both with all corners which do not lie on the boundary of \(X\) or \(P\) smoothed.
Then \(W\) and \(W'\) have exactly one facet which does not lie on the boundary of \(X\) and \(P\), respectively.
Next choose a collar \(k:\phi^{-1}(F_{k+1})-\mathring{W}\times [0,1] \rightarrow C_{k+1}\subset X-\mathring{W}\) of \(\phi^{-1}(F_{k+1})-\mathring{W}\) in \(X-\mathring{W}\).
Here \(\mathring{W}\) denotes \(W\) with the facet, which does not lie on the boundary of \(X\), removed.
Then
\begin{equation*}
  f\circ k \circ (f'^{-1}\times \id_{[0,1]}):((F_{k+1}-f(\mathring{W}))\cap f(W))\times [0,1]\rightarrow (P-f(\mathring{W}))\cap f(W)
\end{equation*}
 is a collar of \((F_{k+1}-f(\mathring{W}))\cap f(W) \) in \((P-f(\mathring{W}))\cap f(W)\).
We can extend this collar to a collar \(k':(F_{k+1}-f(\mathring{W}))\times [0,1]\rightarrow P-f(\mathring{W})\).

Now we can extend \(f\cup f'\) to \(W\cup C_{k+1}\) by \(f\cup k'\circ (f'\times \id_{[0,1]})\circ k^{-1}\).
So we have extended \(f\) to a neighborhood of \(\phi^{-1}(F_1)\cup\dots\cup\phi^{-1}(F_{k+1})\).

By iterating these two steps we may assume that \(f\) is defined on a neighborhood of the boundary of \(X\).

So we may choose neighborhoods \(W\) and \(W'\) of the boundaries of \(X\) and \(P\) such that:
\begin{enumerate}
\item \(W\) and \(W'\) are nice manifolds with corners.
\item \(f\) restricted to \(W\) is a diffeomorphism onto \(W'\).
\item There is a homeomorphism \(h:\partial X\times [0,1]\rightarrow W\) which is a diffeomorphism outside a small neighborhood of \(X_2\times [0,1]\), where \(X_2\) is the union of all faces of \(X\) which have at least codimension two.
Moreover, \(\partial W\) has two components \(h(\partial X \times \{0\})\) and  \(h(\partial X \times \{1\})\).
\item \(X=D^n\cup_{h'}W\) and \(P=D^n\cup_{h''}W'\), where \(h'\) and \(h''\) are diffeomorphisms of \(S^{n-1}\) onto the components  \(h(\partial X \times \{1\})\) and  \(f\circ h(\partial X \times \{1\})\) of \(\partial W\) and \(\partial W'\), respectively.
\end{enumerate}
By the last equality, we have \(P=D^n\cup_{f^{-1}\circ h''} W\).
Moreover, \(X\) and \(P\) are diffeomorphic if \(f^{-1}\circ h''\circ h'^{-1}\) can be extend to a diffeomorphism of \(W\).
If this extension is the identity on the neighborhood of \(\phi^{-1}(F_1)\cup\dots\cup\phi^{-1}(F_k)\), then there is a diffeomorphism \(X\rightarrow P\) which is equal to \(f\) on this neighborhood.

The manifold on which \(f^{-1}\circ h''\circ h'^{-1}\) is defined is diffeomorphic to the standard sphere.
Therefore it can be isotoped to a diffeomorphism \(h'''\) which is the identity on a neighborhood of \(\bigcup_{i=1}^{m-1}F_i\times \{1\}\) and on the set where \(h^{-1}\) is not a diffeomorphism.
Let \(h_1=h|_{\partial X\times \{1\}}\). Then by \(h\circ((h_1^{-1}\circ h'''\circ h_1)\times \id_{[0,1]})\circ h^{-1}\) there is given a diffeomorphism of \(W\) which extends \(h'''\) and is the identity on a neighborhood of  \(\phi^{-1}(F_1)\cup\dots\cup\phi^{-1}(F_k)\).

Therefore the theorem follows.
\end{proof}

\begin{corollary}
\label{sec:class-quas-manif-1}
  Let \(X\) be a nice manifold with corners as in Theorem \ref{sec:class-quas-manif} and \(P\) be a simple polytope of the same dimension as \(X\).
  If there is an isomorphism \(\phi:\mathcal{P}(X)\rightarrow \mathcal{P}(P)\).
  Then there is a diffeomorphism \(f:X\rightarrow P\) such that \(f(F)=\phi(F)\) for each face of \(X\).
\end{corollary}
\begin{proof}
  We prove this statement by induction on the dimension of \(X\).
  If \(X\) is zero- or one-dimensional then there is nothing to prove.
  So assume that \(X\) has dimension at least two.
  Then by the induction hypothesis there is a diffeomorphism of a facet \(F\) of \(X\) onto a facet of \(P\).
  This diffeomorphism can be extended to a diffeomorphism of a collar of \(F\) onto a collar of \(\phi(F)\).
  Now the statement follows from Theorem~\ref{sec:class-quas-manif}.
\end{proof}

\begin{corollary}
  Two combinatorially equivalent polytopes are diffeomorphic as manifolds with corners.
\end{corollary}

Next we want to discuss the question how much the assumptions of Theorem \ref{sec:class-quas-manif} on the diffeomorphism type of the faces of \(X\) may be weakened.
In all dimensions except \(4\) and \(5\) it is known that \(D^n\) has exactly one smooth structure.

Assume that \(X\) is a \(5\)-dimensional nice manifold with corners which is homeomorphic to \(D^5\) such that all faces of codimension at least one of \(X\) are after smoothing the corners diffeomorphic to discs.
If \(X\) is combinatorially equivalent to a polytope \(P\), then it follows from the proof of Theorem \ref{sec:class-quas-manif} that a neighborhood of the boundary of \(X\) is diffeomorphic to a neighborhood of the boundary of \(P\).
Therefore after smoothing the corners the boundary of \(X\) is diffeomorphic to the standard four-dimensional sphere.
Since a manifold homeomorphic to \(D^5\), whose boundary is diffeomorphic to \(S^4\), is diffeomorphic to \(D^5\), it follows that \(X\) is after smoothing the corners diffeomorphic to \(D^5\).

Therefore the assumptions in Theorem \ref{sec:class-quas-manif} on the diffeomorphism type of the faces of \(X\) can be relaxed to:
\begin{itemize}
\item All faces of \(X\) are homeomorphic to a disc.
\item The four-dimensional faces of \(X\) are after smoothing the corners diffeomorphic to the four-dimensional disc.
\end{itemize}

This shows that if one wants to construct exotic smooth structures on a polytope then one has to change the diffeomorphism type of the four-dimensional faces.
Actually we can prove the following theorem.

\begin{theorem}
\label{sec:class-quas-manif-2}
  Let \(P\) be a four-dimensional simple polytope.
  Then there is a bijection between the smooth structures on \(P\) up to diffeomorphism and the smooth structures on \(D^4\) up to diffeomorphism.
\end{theorem}
\begin{proof}
  Let \(P_1,P_2=P\) equipped with different smooth structures.
  Then, by Corollary~\ref{sec:class-quas-manif-1}, there is a diffeomorphism of a facet of \(P_1\) onto a facet of \(P_2\).
  As in the proofs of Corollary~\ref{sec:class-quas-manif-1} and Theorem~\ref{sec:class-quas-manif} we can extend this diffeomorphism to a diffeomorphism of a neighborhood \(W\) of the boundary of \(P_1\) onto a neighborhood of the boundary of \(P_2\).
Let \(D_1\) and \(D_2\) be \(P_1\) and \(P_2\) with corners smoothed, respectively.

Then we have \(P_i=D_i\cup_{f_i}W\), \(i=1,2\), with some diffeomorphisms \(f_i\) of the boundary of \(D_i\) onto a component of the boundary of \(W\).
Since by the Poincar\'e conjecture \(\partial D_i\) is diffeomorphic to \(S^3\) it follows as in the proof of Theorem \ref{sec:class-quas-manif} that \(f_1\circ f_2^{-1}\) extends to a diffeomorphism of \(W\).

Therefore \(P_1\) and \(P_2\) are diffeomorphic if \(D_1\) and \(D_2\) are diffeomorphic.

It remains to prove that for each manifold \(D\) homeomorphic to \(D^4\) there is a smooth structure \(P_1\) on \(P\) such that after smoothing the corners \(P_1\) is diffeomorphic to \(D\).
Such a \(P_1\) can be constructed by removing a disc from the interior of \(P\) and replacing it by \(D\).

Therefore the theorem is proved.
\end{proof}

\begin{theorem}
  Let \(P\) be an \(n\)-dimensional simple polytope with \(n\geq 4\).
  Then there are more than one smooth structures on \(P\) if and only if there are more than one smooth structures on \(D^4\).
\end{theorem}
\begin{proof}
  If \(n=4\), then this follows from Theorem \ref{sec:class-quas-manif-2}.
  Therefore assume \(n>4\).
  If there is only one smooth structure on \(D^4\), then it follows from Theorem \ref{sec:class-quas-manif} and the remarks there after that there is only one smooth structure on \(P\).

  That there are more than one smooth structures on \(P\) if there are more than one smooth structures on \(D^4\) follows from the construction below.
Let \(D\) be a manifold homeomorphic to \(D^4\).
Then there is a five-dimensional nice manifold with corners \(D'\) such that \(D'\) is homeomorphic to \(D^5\) and has exactly two facets one of them diffeomorphic to \(D\) the other diffeomorphic to \(D^4\).
Denote \(D'\times \Delta^{n-5}\cup_{D\times \Delta^{n-5}} D\times \Delta^{n-4}\) by \(S_D\).
Then the diffeomorphism type of \(D^4\times\Delta^{n-5}\cup \partial D\times \Delta^{n-4}\subset S_D\) is independent of \(D\).
Moreover, \(S_{D^4}\) is diffeomorphic to \(D^4\times \Delta^{n-4}\) with some corners smoothed.

Now embed \(D^4\) into the interior of a four-dimensional face of \(P\).
Then this embedding extends to an embedding of \(S_{D^4}\).
Now remove \(S_{D^4}\) from \(P\) and replace it by \(S_D\).
The resulting manifold with corners is face-preserving homeomorphic to \(P\) and has a four-dimensional face which is after smoothing the corners diffeomorphic to \(D\).
Therefore it cannot be diffeomorphic to \(P\) if \(D\) is not diffeomorphic to \(D^4\).
\end{proof}

Now we turn to the question how many equivariant smooth structures there are on a quasitoric manifold.

Let \(P\) be a simple polytope and \(M\) be a quasitoric manifold over \(P\).
Denote by \(\lambda\) the map which assigns to each point \(x\) of \(P\) the isotropy group of a point in the \(T\)-orbit in \(M\) corresponding to \(x\).
Then \(\lambda\) is constant on the interior of each face \(F\) of \(P\).
And we define \(\lambda(F)=\lambda(x)\) for \(x\) an interior point of \(F\).

It has been shown by Davis and Januszkiewicz \cite[Proposition 1.8, p. 424]{0733.52006} that \(M\) is equivariantly homeomorphic to \(M(P,\lambda)=(P\times T)/\!\!\sim\) with \((x_1,t_1)\sim (x_2,t_2)\) if and only if \(x_1=x_2\) and \(t_1 t_2^{-1}\in \lambda(x_1)\). 
Here we want to show that the smooth structures on \(P\) are one-to-one to the equivariant smooth structures on \(M(P,\lambda)\).

For the proof we use the theory of \(T\)-normal systems developed by Davis in \cite{0403.57002}.
There it has been shown that the isomorphism types of \(T\)-normal systems are one-to-one to the diffeomorphism types of \(T\)-manifolds.

A \(T\)-normal system corresponding to a quasitoric manifold \(M\) (or more generally to a locally standard torus manifold) consists out of the following data (we use the same notation as in \cite[Definition 4.1, p.352]{0403.57002}):
\begin{enumerate}
\item A closed set \(J\) of normal \(T\)-orbit types.
\item For each \(\alpha\in J\) a principal \(T\)-bundle \(P_\alpha\) over a manifold with \(J_\alpha\)-faces \(B_\alpha\). Here \(J_\alpha=\{\beta\in J;\; \beta<\alpha\}\).
\item For each pair \((\alpha,\beta)\in J\times J\) with \(\beta>\alpha\), an isomorphism of \(T\)-bundles:
  \begin{equation*}
    \theta_{\alpha,\beta}:\partial_\alpha P^\alpha_\beta\times_T P_\alpha\rightarrow \partial_\alpha P_\beta.
  \end{equation*}
These isomorphisms satisfy a certain compatibility condition.
\end{enumerate}

The set of base spaces \(\{B_\alpha;\;\alpha\in J\}\) together with the maps induced by the \(\theta_{\alpha,\beta}\) on the base spaces form a \(B\)-normal system which descripes the diffeomorphism type of \(M/T\).

At first we indicate how to construct from a polytope \(P\) (equipped with some smooth structure) and a characteristic map \(\lambda\) a quasitoric manifold \(M\) such that \(M\) is equivariantly homeomorphic to \(M(P,\lambda)\) and \(M/T\) is diffeomorphic to \(P\).

At first note that the normal orbit type of a point in a quasitoric manifold \(M\) is completely determined by the map \(\lambda\).
Therefore we take as our set \(J\) the set of these normal orbit types.

We construct the manifolds \(B_\alpha\) by removing minimal strata from \(P\), repeatedly.
The \(B_\alpha\) constructed in this way are contractible.
Therefore all the \(T\)-bundles \(P_\alpha\) are trivial.
Moreover, from this iteration we get diffeomorphisms \(\theta'_{\alpha,\beta}:\partial_\alpha B^{\alpha}_\beta\times B_\alpha\rightarrow \partial_\alpha B_\beta\).

We define \(\theta_{\alpha,\beta}\) by the composition of the natural isomorphism
  \(\partial_\alpha P^\alpha_\beta\times_T P_\alpha\rightarrow\partial_\alpha B^\alpha_\beta\times B_\alpha\times T\) and \(\theta'_{\alpha,\beta}\times \id_T\).
Then the compatibility conditions are automatically satisfied.

This \(T\)-normal system corresponds to a quasitoric manifold homeomorphic to \(M(P,\lambda)\).

Next we want to show that two quasitoric manifolds \(M_1\) and \(M_2\) are equivariantly  diffeomorphic if both manifolds are equivariantly homeomorphic to \(M(P,\lambda)\) and the \(T\)-actions on \(M_1\) and \(M_2\) induce the same differentiable structure on \(P\).

To do so we consider the \(T\)-normal systems corresponding to \(M_1\) and \(M_2\) and show that they are isomorphic.

All parts of these \(T\)-normal systems are the same and as described above.
The only exception are the isomorphisms \(\theta_{\alpha,\beta}\).

Denote the isomorphisms corresponding to \(M_1\) by \(\theta_{\alpha,\beta}\) and the isomorphisms corresponding to \(M_2\) by \(\theta'_{\alpha,\beta}\).
Since the orbit spaces of \(M_1\) and \(M_2\) are diffeomorphic, the \(T\)-normal systems corresponding to \(M_1\) and \(M_2\) induce isomorphic \(B\)-normal systems.
Therefore we may assume that \(\theta_{\alpha,\beta}\) and \(\theta'_{\alpha,\beta}\) induce the same maps on the base spaces.

We have to find bundle isomorphisms \(\phi_\alpha: P_\alpha\rightarrow P_\alpha\) such that the following diagram commutes
\begin{equation}
\label{eq:app:1}
  \xymatrix{
    \partial_\alpha P_\beta^\alpha \times_T P_\alpha \ar[r]^{\times \phi_\alpha} \ar[d]_{\theta_{\alpha,\beta}}&\partial_\alpha P_\beta^\alpha \times_T P_\alpha \ar[d]^{\theta'_{\alpha,\beta}}\\
    \partial_\alpha P_\beta\ar[r]_{\phi_\beta}&\partial_{\alpha}P_\beta\\}
\end{equation}

We construct the \(\phi_\alpha\) inductively starting with isomorphisms on the minimal strata.
The maps which we will construct in this way will always induce the identity on the base spaces.
When we pass from one stratum to the next one then the upper part of the diagram~(\ref{eq:app:1}) is already defined.
This implies that \(\phi_\beta\) is defined on the boundary of \(P_\beta\) and we have to extend this isomorphism to all of \(P_\beta\).
Note that the isomorphisms of a trivial \(T\)-bundle over the space \(B\) correspond one-to-one to maps \(B\rightarrow T\).
Since the boundary of \(B_\beta\) is always homeomorphic to a sphere, there are now obstructions to extend \(\phi_\beta\) to all of \(P_\beta\) if \(\dim B_\beta\) is greater or equal to three.

So that we only have to define \(\phi_\alpha\) for \(\dim B_\alpha \leq 2\).
At first choose isomorphisms \(P_\alpha\rightarrow P_\alpha\) for all \(\alpha\) with \(\dim B_\alpha=0\).

Next choose a line shelling \(F_1,\dots,F_m\) of \(P\).
Assume that the isomorphisms \(\phi_\alpha\) are defined for all strata up to dimension two in \(\bigcup_{i=1}^kF_i\).
From the shelling we get an ordering of the two dimensional faces \(G_1,\dots,G_{m'}\) of \(F_{k+1}\) such that there are two cases:
\begin{enumerate}
\item \(G_i\cap(\bigcup_{j=1}^{i-1}G_j)\) is contractible.
\item \(G_i\cap(\bigcup_{j=1}^{i-1}G_j)=\partial G_i\) and there is a three dimensional face \(H\) of \(F_{k+1}\) such that \(G_i\subset \partial H \subset \bigcap_{j=1}^iG_j\).
\end{enumerate}
Moreover there is an \(1\leq l\leq m'\) such that \(G_1,\dots,G_l\) are the two-dimensional faces in \(\bigcap_{i=1}^k F_i\cap F_{k+1}\).

Denote by \(\alpha_{G_i}\) the the normal orbit type of points in \(M_1\), \(M_2\) above \(G_i\).
Then \(B_{\alpha_{G_i}}\) is diffeomorphic to \(G_i\) with all vertices cut off.
If \(\phi_{\alpha_{G_i}}\) is defined then this defines also an isomorphism \(\phi_{\alpha_{G'}}\) for \(\dim G' =1\) and \(G'\subset G_i\) because in this case \(\partial_{\alpha_{G'}}B^{\alpha_{G'}}_{\alpha_{G_i}}\) is just one point.

We construct the \(\phi_{\alpha_{G_i}}\)'s inductively by starting with \(G_1\). Assume that the \(\phi_{\alpha_{G_j}}\) are already defined for \(j=1,\dots,i-1\).
When we are in case one then there are no obstructions to define \(\phi_{\alpha_{G_i}}\).
If we are in case two then we have an isomorphism \(\phi: \partial P_{\alpha_{G_i}}\rightarrow \partial P_{\alpha_{G_i}}\) which we want to extend to all of \(P_{\alpha_{G_i}}\).
We identify \(P_{\alpha_{G_i}}\) with a part of the boundary of \(P_{\alpha_{H}}\).
Then \(\phi\) extends to an isomorphism of \(\partial P_{\alpha_H}-P_{\alpha_{G_i}}\) because this complement is a union of the \(\partial_{\alpha_G}P^{\alpha_G}_{\alpha_H}\times_T P_{\alpha_G}\) with \(G\subset H \cap \bigcup_{j=1}^{i-1}G_i\).
Therefore the map from \(\partial B_{\alpha_{G_i}}\rightarrow T\) corresponding to \(\phi\) is null-homotopic.
Hence, \(\phi\) extends to an isomorphism of \(P_{G_i}\).

To conclude this section we state what we have proved:

\begin{theorem}
  \label{sec:smooth-struct-quas}
  Let \(P\) be a polytope and \(\lambda\) a characteristic map defined on \(P\). Then the \(T\)-equivariant smooth structures on \(M(P,\lambda)\) (up to equivariant diffeomorphisms) are one-to-one to the smooth structures on \(P\) (up to diffeomorphisms which are compatible with \(\lambda\)).
\end{theorem}

By combining Theorem~\ref{sec:smooth-struct-quas} and Corollary~\ref{sec:class-quas-manif} we get the following corollary.

\begin{corollary}
  Let \(M\) and \(M'\) be two strongly quasitoric manifolds over polytopes \(P\) and \(P'\), respectively.
  If \(P\) and \(P'\) are combinatorially equivalent and \(M\) and \(M'\) give rise to the same characteristic maps, then \(M\) and \(M'\) are equivariantly diffeomorphic.
\end{corollary}

\section{Six-dimensional torus manifolds}
\label{sec:class}

In this section we discuss invariants of simply connected six-dimensional torus manifolds \(M\) with \(H^{{odd}}(M)=0\) which determine \(M\) up to equivariant diffeomorphism.
The results of this section are used in section \ref{sec:tori_diff_six} where the number of conjugacy classes of three-dimensional tori in the diffeomorphism groups of some six-dimensional torus manifolds is determined.

Before we state our first theorem we recall some of the results of \cite{1111.57019}.

If \(M\) is a six-dimensional torus manifold with \(H^{{odd}}(M)=0\), then, by Theorem 4.1 of  \cite[p. 720]{1111.57019}, \(M\) is locally standard and the orbit space \(M/T\) is a nice manifold with corners.
Moreover, all faces of \(M/T\) are acyclic.

We denote by \(\mathcal{P}\) the face poset of \(M/T\) and by 
\begin{equation*}
\lambda:\{\text{faces of } M/T\}\rightarrow \{\text{subtori of } T\}  
\end{equation*}

 the map which assigns to each face of \(M/T\) the isotropy group of an interior point of this face.
This map is called the characteristic map of \(M\).

\begin{theorem}
\label{sec:class-six-dimens}
  Let \(M\) be a six-dimensional torus manifold with \(H^{{odd}}(M)=0\) and \(\pi_1(M)=0\). Then the equivariant diffeomorphism type of \(M\) is uniquely determined by \(\mathcal{P}\) and \(\lambda\).
\end{theorem}

For the proof of Theorem \ref{sec:class-six-dimens} we need two lemmas which we prove before proving the theorem.

\begin{lemma}
\label{sec:six-dimens-torus}
  Let \(X\) and \(X'\) be two three-dimensional nice manifold with corners which are homeomorphic to \(D^3\) such that all of their two-dimensional faces are homeomorphic to \(D^2\).
  Then \(X\) and \(X'\) are diffeomorphic if and only if their face posets are isomorphic.
\end{lemma}
\begin{proof}
At first note that since \(D^2\) and \(D^3\) have unique differentiable structures, all faces of dimension greater than one of \(X\) and \(X'\) are after smoothing the corners diffeomorphic to discs.

Therefore there are two cases:
\begin{enumerate}
\item \(X\) has two two-dimensional faces and one one-dimensional face which is diffeomorphic to \(S^1\).
\item All faces of \(X\) are diffeomorphic to discs.
\end{enumerate}
Since every diffeomorphism of \(S^1\) extends to a diffeomorphism of \(D^2\), one can argue (in both cases) as in the proofs of Theorem~\ref{sec:class-quas-manif} and Corollary \ref{sec:class-quas-manif-1} to prove this lemma.
Here instead of the shelling of the polytope one can use any ordering of the facets of \(X'\).
\end{proof}

\begin{lemma}
\label{sec:six-dimens-torus-1}
  Let \(X\) be a three-dimensional nice manifolds with corners as in the previous lemma. Then there is an ordering \(F_1,\dots,F_m\) of the facets of \(X\) such that for all \(k<m\) all components of \(\bigcup_{i=1}^k F_i\) are homeomorphic to discs.
\end{lemma}
\begin{proof}
  Let \(F_m\) be any facet of \(X\).
  Then \(Y=\partial X-F_m\) is a two-dimensional disc.
  Therefore the lemma follows from the following claim:

Let \(Y=\bigcup_{i=1}^m Y_i\) be a two-dimensional disc, such that \(Y_i\subset X\) are discs with piecewise differentiable boundaries and such that \(Y_i\cap Y_j\subset \partial Y_i\) for all \(i,j\).
Then there exists an ordering \(Y_{i_1},\dots,Y_{i_m}\) such that for \(k\leq m\) all components of \(\bigcup_{j=1}^k Y_{i_j}\) are discs.

We prove this claim by induction on \(m\).
Let \(Y_{i_m}\) such that the intersection of \(Y_{i_m}\) with the boundary of \(Y\) is non-empty.
Then the complement of the interior of \(Y_{i_m}\) in \(Y\) is a disjoint union of discs.
Therefore we get the ordering of the \(Y_i\) from the induction hypothesis.  
\end{proof}

\begin{proof}[Proof of Theorem 6.1]
  Since \(M\) is simply connected.
  The orbit space \(X=M/T\) is simply connected.
  Therefore it is a disc.
  Hence, we see with Lemma~\ref{sec:six-dimens-torus} that the diffeomorphism type of \(X\) depends only on \(\mathcal{P}\).
  
  Moreover, we can argue as in section \ref{sec:app-intro} to see that the equivariant diffeomorphism type of \(M\) depends only on the smooth structure on \(X\) and the map \(\lambda\).
  Here one has to use the ordering of the facets of \(X\) from Lemma~\ref{sec:six-dimens-torus-1} instead of the shelling of the polytopes.

  If we put these two steps together then the theorem follows.
\end{proof}

Since \(\mathcal{P}\) and \(\lambda\) are determined by the \(T\)-equivariant homeomorphism type of \(M\) we get:

\begin{corollary}
\label{sec:six-dimens-torus-2}
  Simply connected six-dimensional torus manifolds with vanishing odd degree cohomology are equivariantly homeomorphic if and only if they are equivariantly diffeomorphic.
\end{corollary}

  With  Corollary 7.8 of \cite[p. 736]{1111.57019} and Theorem \ref{sec:class-six-dimens} one can show that Theorem 2.2 of \cite{wiemeler-remarks} holds for simply connected six-dimensional torus manifolds with cohomology generated in degree two.
To be more precise, we have the following theorem.

\begin{theorem}
  \label{sec:class-six-dimens-2}
  Let \(M,M'\) be simply connected torus manifolds of dimension six such that \(H^*(M)\) and \(H^*(M')\) are generated by their degree two parts.
Let \(m,m'\) be the numbers of characteristic submanifolds of \(M\) and \(M'\), respectively.
Assume that \(m\leq m'\).
  Furthermore, let \(u_1,\dots,u_m\in H^2(M)\) be the Poincar\'e-duals of the characteristic submanifolds of \(M\) and  \(u_1',\dots,u_{m'}'\in H^2(M')\) the Poincar\'e-duals of the characteristic submanifolds of \(M'\).
  If there is a ring isomorphism \(f:H^*(M)\rightarrow H^*(M')\) and a permutation \(\sigma:\{1,\dots,m'\}\rightarrow \{1,\dots,m'\}\) with \(f(u_i)=\pm u_{\sigma(i)}'\), \(i=1,\dots,m\), then \(M\) and \(M'\) are weakly T-equivariantly diffeomorphic.
\end{theorem}

\begin{theorem}
  Let \(M\) be a six-dimensional torus manifold with \(H^{odd}(M)=0\).
Then there are a three-dimensional homology sphere \(X\) and a simply connected torus manifold \(M'\) such that \(M=\alpha(M',X)\).
\end{theorem}
\begin{proof}
The orbit space of the \(T\)-action on \(M\) is after smoothing the corners a three-dimensional homology disc \(\tilde{D}^3\).
Therefore we have a decomposition of \(M\) of the form
  \begin{equation*}
    M=\tilde{D}^3\times T^3\cup_{f} M'',
  \end{equation*}
  where \(M''\) is a six-dimensional \(T\)-manifold with boundary and \(M''^T\neq\emptyset\) and \(f:\partial\tilde{D}^3\times T^3\rightarrow \partial M''\) is an equivariant diffeomorphism.
  The orbit space of \(M''\) is after smoothing the corners diffeomorphic to \(S^2\times [0,1]\).

  Choose a diffeomorphism \(h:S^2\rightarrow \partial \tilde{D}^3\).
  Then let \(M'\) be the torus manifold
  \begin{equation*}
    M'=D^3\times T^3\cup_{f\circ(h\times \id_{T^3})}M''
  \end{equation*}
and \(X=\tilde{D}^3\cup_{h}D^3\).
Then, by Lemma \ref{sec:fund-groups-torus-6}, we have \(\pi_1(M')=\pi_1(D^3)=1\) and \(X\) is a homology sphere.

It is immediate from the definitions of \(M'\) and \(X\) that \(M=\alpha(M',X)\).
\end{proof}

\section{Tori in the diffeomorphism group of high dimensional torus manifolds}
\label{sec:tori_in_diff}

In this section we prove the following theorem.

\begin{theorem}
  \label{sec:tori-diff-group}
  Let \(M\) be a \(2n\)-dimensional torus manifold, \(n\geq 4\).
  Then there are infinitely many non-conjugated \(n\)-dimensional tori in \(\Diff(M)\).
\end{theorem}

\begin{remark}
  The number of tori in the diffeomorphism group of four-dimensio\-nal simply connected torus manifolds were computed by Melvin \cite{0486.57016}.
  He proved that there are four-dimensional torus manifolds which have only a finite number of non-conjugated tori in their diffeomorphism group.
\end{remark}

For the proof of Theorem~\ref{sec:tori-diff-group} we need two lemmas.
Let \(M\) be a torus manifold. A characteristic submanifold of \(M\) is a connected submanifold \(N\) of codimension two which is fixed by a one-dimensional subtorus of \(T\) such that \(N^T\) is non-empty.
It is easy to see that each torus manifold has a finite number of characteristic submanifolds.
Moreover, the isotropy group of a generic point in a characteristic manifold is connected.

\begin{lemma}
\label{sec:tori-diff-group-1}
  Let \(M\) be a \(2n\)-dimensional torus manifolds and \(T/S^1\) an \((n-1)\)-dimensional orbit, which is contained in a characteristic submanifold \(M_1\) of \(M\).
Furthermore, denote by \(T'\) a \((n-1)\)-dimensional subtorus of \(T\) such that \(T=T'\times S^1\).
  Then the \(T\)-equivariant normal bundle of \(T/S^1\cong T'\) is given by
  \begin{equation*}
    N(T/S^1,M)=T'\times \R^{n-1}\times \C,
  \end{equation*}
where \(T\) acts on \(\C\) through the projection \(T=T'\times S^1\rightarrow S^1\hookrightarrow \C^*\).
\end{lemma}
\begin{proof}
  For \(x\in T/S^1\) the \(S^1\)-representation \(N_x(M_1,M)\) is isomorphic to \(\C\).
Therefore we have
\begin{equation*}
  N(M_1,M)|_{T/S^1}=T\times_{S^1}\C = T'\times \C
\end{equation*}
and 
\begin{equation*}
  N(T/S^1,M_1)=T/S^1\times \R^{n-1}
\end{equation*}
because \(T/S^1\) is a principal orbit of the \(T\)-action on \(M_1\).
Therefore the statement follows.
\end{proof}

Similarly one proves.
\begin{lemma}
\label{sec:tori-diff-group-2}
  Let \(X\) be a closed connected oriented manifold of dimension \(n+1\geq 3\) on which \(S^1\) acts effectively, such that \(X^{S^1}\) is connected and \(\codim X^{S^1}=2\).
  Then for \(x\in X^{S^1}\) one has
  \begin{equation*}
    T_xX=\R^{n-1}\times \C.
  \end{equation*}
\end{lemma}

These two lemmas enable us to introduce a second construction \(\beta(M,X)\).
Let \(M\), \(T/S^1\), \(X\), \(x\) as in Lemmas \ref{sec:tori-diff-group-1} and \ref{sec:tori-diff-group-2}.
Then we define the torus manifold \(\beta(M,X)\) as
\begin{equation*}
  \beta(M,X)= M-(T'\times \mathring{D}_1^{n+1})\cup_{T'\times S^n} T'\times (X- \mathring{D}^{n+1}_2),
\end{equation*}
where \(T'\times D_1^{n+1}\) is an equivariant tubular neighborhood of \(T/S^1\) in \(M\) and \(D_2^{n+1}\) is an equivariant tubular neighborhood of \(x\) in \(X\).

It is easy to see that if \(X\) is diffeomorphic to \(S^{n+1}\) then \(\beta(M,X)\) and \(M\) are diffeomorphic.
Moreover, if \(M_1,\dots,M_m\) are the characteristic submanifolds of \(M\), then \(\alpha(M_1,X^{S^1}),M_2,\dots,M_m\) are the characteristic submanifolds of \(\beta(M,X)\).

Now we construct \(S^1\)-manifolds with the properties mentioned in Lemma \ref{sec:tori-diff-group-2} which are diffeomorphic to the standard sphere.
This construction is analogous to a construction in \cite{0139.16903}.

Let \(Y\) be a contractible compact manifold of dimension \(n\geq 4\).
Then \(S^1\) acts differentiable on \(D^2\times Y\) (with corners equivariantly smoothed).
Moreover, \(X=\partial (D^2\times Y)= S^1\times Y\cup_{S^1\times \partial Y} D^2\times \partial Y\) is simply connected and \(D^2\times Y\) is contractible.
Therefore it follows from the h-cobordism theorem that \(D^2\times Y\) is diffeomorphic to the standard \(n+2\)-dimensional disc \cite[Corollary 4.5, p. 154]{0767.57001}.

Therefore \(X\) is the standard sphere and \(X^{S^1}=\partial Y\).

Since \(n\geq 4\), there are infinitely many contractible manifolds \(\{Y_i\}\) with
\begin{equation*}
  \rank \pi_1(\partial Y_i)\neq \rank \pi_1(\partial Y_j)
\end{equation*}
 for \(i\neq j\).
By the above construction, the \(Y_i\) give rise to an infinite series of \(n\)-dimensional tori in the diffeomorphism group of \(M\).
If two of these tori are conjugated in the diffeomorphism group of \(M\), then the disjoint unions of the characteristic submanifolds corresponding to these tori are diffeomorphic.
But this is impossible by Theorem~\ref{sec:fund-groups-torus-2} and the remark about the characteristic submanifolds of \(\beta (M, X)\).

Therefore Theorem~\ref{sec:tori-diff-group} follows.

\begin{remark}
If \(H^*(M)\) is generated by its degree two part, the cohomology ring of \(M\) may be described in terms of the combinatorial structure of its characteristic submanifolds \cite[Corollary 7.8, p. 735-736]{1111.57019}.
Since this combinatorial structure does not change if we replace \(M\) by \(\beta(M,X)\), there is an isomorphism
\begin{equation*}
  H^*(M)\rightarrow H^*(\beta(M,X))
\end{equation*}
which preserves the Poincar\'e duals of the characteristic submanifolds of \(M\) and \(\beta(M,X)\).
This shows that Theorem \ref{sec:class-six-dimens-2} does not hold for torus manifolds of dimension greater than six.  
\end{remark}

\section{Tori in the diffeomorphism group of six-dimensional torus manifolds}
\label{sec:tori_diff_six}

In this section we compute the number of conjugacy classes of three-dimensional tori in the diffeomorphism group of some six-dimensional torus manifolds.
We summarize the results of this section in the following table.

\begin{center}
  \begin{tabular}{|p{7cm}|c|}
    \(M\)& tori in \(\Diff(M)\)\\\hline\hline
    \(S^6\)& \(1\)\\\hline
    \(\C P^3\) & \(1\)\\\hline
    \(S^2\times N\), \(N\) simply connected \(4\)-dimensional torus manifold, \(N\neq k \C P^2\), \(k\geq 2\)&\(\infty\)\\\hline
    \(S^2\times k \C P^2\), \(k\geq 2\)& as in \(\Diff(k\C P^2)\)\\
  \end{tabular}
\end{center}

We should note that this table shows together with Melvin's results \cite{0486.57016} that the number of conjugacy classes of tori in the diffeomorphism group of six-dimensional torus manifolds might be one, an arbitrary large finite number or infinite.

\begin{theorem}
  Let \(M\) be a six-dimensional torus manifold which is diffeomorphic to \(\C P^3\) or \(S^6\). Then \(M\) is weakly equivariantly diffeomorphic to \(\C P^3\) or \(S^6\), respectively, with \(T^3\) acting in the standard way.
\end{theorem}
\begin{proof}
  We give the proof only for the case that \(M\) is diffeomorphic to \(\C P^3\).
  The proof for \(S^6\) is similar.
  
  By Theorem 5.1 of \cite[p. 393]{0246.57017} and \cite{0216.20202}, all characteristic submanifolds of \(M\) are diffeomorphic to \(\C P^2\).
  Therefore all facets of \(M/T\) are triangles.
  Because \(M/T\) has \(\chi(M)=4\) vertices, it follows that \(M/T\) is a tetrahedron.
  Because up to automorphism of \(T^3\) there is only one way to define a characteristic map \(\lambda\), the statement follows from Theorem~\ref{sec:class-six-dimens}.
\end{proof}

\begin{lemma}
\label{sec:tori-diff-group-3}
  Let \(M\) and \(M'\) be simply connected four-dimensional torus manifolds such that \(\chi(M)=\chi(M')\).
  Then \(M\times S^2\) and \(M'\times S^2\) are weakly equivariantly diffeomorphic if and only if \(M\) and \(M'\) are weakly equivariantly diffeomorphic.
\end{lemma}
\begin{proof}
  We first consider the case where \(\chi(M)=\chi(M')\neq 4\).
  Then \(M/T=M'/T\) is not diffeomorphic to the square \(I^2=[0,1]^2\).
  
  A weakly equivariant diffeomorphism \(h:M\times S^2\rightarrow M'\times S^2\) induces a diffeomorphism \(\bar{h}:M/T\times I\rightarrow M'/T\times I\).
  Since \(M/T\) is not diffeomorphic to \(I^2\), there are vertices \(p_1,p_2\) of \(I\) such that \(\bar{h}(M/T\times\{p_1\}) = M'/T\times \{p_2\}\).

Therefore \(h\) restricts to a weakly equivariant diffeomorphism of \(M=\pi^{-1}(M/T\times\{p_1\})\rightarrow \pi^{-1}(M'/T\times\{p_2\})=M'\).

If \(\chi(M)=4\), then \(M/T=M'/T=I^2\) and \(M\times S^2/T=M'\times S^2/T=I^3\).
The characteristic maps of \(M\) and \(M'\) are given as in figure \ref{fig:2}, see also \cite[p. 427]{0733.52006}. Here \(p\) is an integer.
Moreover, in that figure we have identified the one-dimensional torus \(\lambda(F)\), \(F\) facet of \(M/T\), with a vector in \(\R^3=LT^3\) tangent to \(\lambda(F)\).

\begin{figure}
  \centering
  \includegraphics{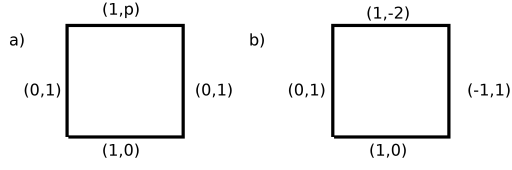}
  \caption{The orbit spaces and characteristic maps of four-dimensional simply connected torus manifolds with Euler characteristic 4.}
  \label{fig:2}
\end{figure}
Assume that there is an weakly equivariant diffeomorphism  \(h:M\times S^2\rightarrow M'\times S^2\).

If the characteristic map of \(M\) is of type b), then there is one pair of opposite facets of \(M\times S^2/T\) on which the characteristic map takes the same value.
If the characteristic map of \(M\) is of type a) and \(p\neq 0\), then there are two pairs of opposite facets of \(M\times S^2/T\) on which the characteristic map takes the same value.
If the characteristic map of \(M\) is of type a) and \(p= 0\), then there are three pairs of opposite facets of \(M\times S^2/T\) on which the characteristic map takes the same value.

Therefore we may assume that both characteristic maps of \(M\) and \(M'\) are of type a) or both characteristic maps are of type b).
In the second case \(M\) and \(M'\) are weakly \(T\)-equivariantly diffeomorphic.

Hence, we may assume that both characteristic maps are of type a).
If for one of the manifolds \(M\) and \(M'\) the value of \(p\) is zero, then this is also the case for the other manifold.
Hence, they are weakly equivariantly diffeomorphic.

So assume that the values \(p_M\) and \(p_{M'}\) of \(p\) for \(M\) and \(M'\) are both non-zero.

The diffeomorphism \(h\) induces an diffeomorphism \(\bar{h}:I^3\rightarrow I^3\) and there is an automorphism \(\psi\) of \(T^3\) such that for each facet \(F\) of \(I^3\) we have
\begin{equation*}
  \lambda_{M'}(\bar{h}(F))=\psi(\lambda_M(F)).
\end{equation*}
Here \(\lambda_{M}\) and \(\lambda_{M'}\) are the characteristic maps of \(M\) and \(M'\), respectively.

Let \(L\psi\) be the automorphism of \(LT^3=\R^3\) induced by \(\psi\). Then we have
\begin{align*}
  L\psi((0,1,0))&=\pm (0,1,0),\\
  L\psi((0,0,1))&=\pm (0,0,1),\\
  L\psi(\{\pm(1,0,0),\pm(1,p_M,0)\})&=\{\pm(1,0,0),\pm(1,p_{M'},0)\},
\end{align*}
or
\begin{align*}
  L\psi((0,1,0))&=\pm (0,0,1),\\
  L\psi((0,0,1))&=\pm (0,1,0),\\
  L\psi(\{\pm(1,0,0),\pm(1,p_M,0)\})&=\{\pm(1,0,0),\pm(1,p_{M'},0)\}.
\end{align*}
Since \(L\psi((1,p_M,0))\in\langle L\psi((1,0,0)), L\psi((0,1,0))\rangle\), we must be in the first case.
Therefore \(\bar{h}\) maps the facets \(F_1,F_2\) of \(I^3\) with \(\bar{\lambda}_M(F_1)=\bar{\lambda}_M(F_2)=(0,0,1)\) to the facets \(F_1',F_2'\) of \(I^3\) with \(\bar{\lambda}_{M'}(F_1')=\bar{\lambda}_{M'}(F_2')=(0,0,1)\).
Here \(\bar{\lambda}(F)\) indicates a vector in \(\R^3=LT\) which spans the Lie-algebra of \(\lambda(F)\).
The preimages of these facets under the orbit maps are equivariantly diffeomorphic to \(M\) and \(M'\), respectively.
Therefore \(h\) restricts to a weakly equivariant diffeomorphism of \(M\) and \(M'\).
\end{proof}

\begin{theorem}
  Let \(N\) be a simply connected four-dimensional torus manifold, \(N\neq k\C P^2\), \(k\geq 2\), then there are infinitely many conjugacy classes of three-dimensional tori in \(S^2\times N\).
\end{theorem}
\begin{proof}
  Two three-dimensional tori in the diffeomorphism group of \(S^2\times N\) are conjugated if and only if they give rise to weakly equivariantly diffeomorphic torus actions on \(S^2\times N\).
  An analogous statement holds for two-dimensional tori in the diffeomorphism group of \(N\).
  Therefore, if \(N\neq S^4,\C P^2\), the theorem follows from Lemma~\ref{sec:tori-diff-group-3} and the fact that there are infinitely many conjugacy classes of tori in the diffeomorphism group of \(N\) \cite{0486.57016}.

Now we consider the cases \(N=S^4\) and \(N=\C P^2\).
Let \(\gamma\) be the canonical line bundle over \(\C P^1=S^2\) and \(\bar{\gamma}\) its dual.
Then there is  exactly one lift of the \(S^1\)-action on \(\C P^1\) to \(\gamma\) such that \(\gamma_{(1:0)}\) is the trivial \(S^1\)-represenation.
Then the weight of the \(S^1\)-representation \(\gamma_{(0:1)}\) is one.

Through the bundle isomorphism
\begin{equation*}
  \C P^1\times \C^2\cong \otimes_{i=1}^a \gamma \oplus \otimes_{i=1}^a \bar{\gamma},
\end{equation*}
\(a\geq 0\), there is given a \(S^1_1\times S^1_2\times S^1\)-action on \(\C P^1\times \C^2\), where \(S^1_1\) acts by multiplication on \(\gamma\) and \(S^1_2\) acts by multiplication on \(\bar{\gamma}\).

This action induces a torus action on
\begin{equation*}
  \C P^1 \times S^4 = S(\otimes_{i=1}^a \gamma \oplus \otimes_{i=1}^a \bar{\gamma})
\end{equation*}
and
\begin{equation*}
  \C P^1\times \C P^2= P(\otimes_{i=1}^a \gamma \oplus \otimes_{i=1}^a \bar{\gamma}\oplus \C).
\end{equation*}

\begin{figure}
  \centering
  \includegraphics{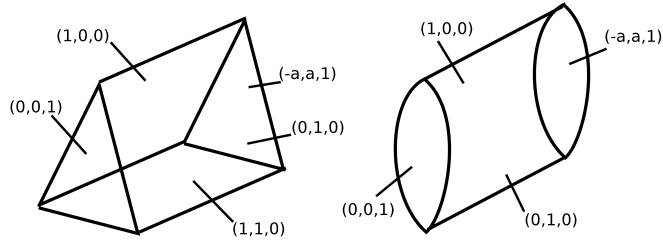}
  \caption{The orbit spaces and characteristic maps of $P(\otimes_{i=1}^a \gamma \oplus \otimes_{i=1}^a \bar{\gamma}\oplus \C)$ and $S(\otimes_{i=1}^a \gamma \oplus \otimes_{i=1}^a \bar{\gamma})$.}
  \label{fig:1}
\end{figure}
The orbit spaces and the characteristic maps of these actions are shown in figure \ref{fig:1}.
In that figure we have identified the one-dimensional torus \(\lambda(F)\), \(F\) facet of $P(\otimes_{i=1}^a \gamma \oplus \otimes_{i=1}^a \bar{\gamma}\oplus \C)/T$ or $S(\otimes_{i=1}^a \gamma \oplus \otimes_{i=1}^a \bar{\gamma})/T$, with a vector in \(\R^3=LT^3\) tangent to \(\lambda(F)\).

By considering the characteristic maps for these actions, one finds that two actions defined in this way by different \(a\)'s are not weakly equivariantly diffeomorphic.

We give here the details of the argument for \(\C P^1\times S^4\).
The argument for \(\C P^1\times \C P^2\) is similar.

Assume that \(M_a=S(\otimes_{i=1}^a \gamma \oplus \otimes_{i=1}^a \bar{\gamma})\) and \(M_b=S(\otimes_{i=1}^b \gamma \oplus \otimes_{i=1}^b \bar{\gamma})\) are weakly equivariantly diffeomorphic.
Then there are a diffeomorphism \(f:M_a/T\rightarrow M_b/T\) and a automorphism \(\psi\) of \(T\) such that for each facet \(F\) of \(M_a/T\) we have \(\psi(\lambda_a(F))=\lambda_b(f(F))\).
Here \(\lambda_a\) and \(\lambda_b\) denote the characteristic maps of \(M_a\) and \(M_b\), respectively.
Since, for each facet \(F\) of \(M_a\), \(f(F)\) and \(F\) have the same number of vertices, we see from figure \ref{fig:1} that \(\psi\) induces an automorphism \(L\psi\) of the Lie algebra \(LT=\R^3\) of \(T\) such that
\begin{align*}
  L\psi(\{\pm (0,0,1), \pm (-a,a,1)\})&=\{\pm (0,0,1), \pm (-b,b,1)\},\\
  L\psi(\{\pm (1,0,0), \pm (0,1,0)\})&=\{\pm (1,0,0), \pm (0,1,0)\}.
\end{align*}

From these two equations we see that
\begin{align*}
  (\pm a, \pm a,0)&= -a \cdot L\psi((1,0,0)) + a\cdot L\psi((0,1,0))\\ 
  &= L\psi((-a,a,1))-L\psi((0,0,1))=(\pm b, \pm b, c)
\end{align*}
with \(c\in \{0,2\}\).
Since \(a,b\geq 0\), it follows that \(a=b\) and \(c=0\).
Therefore the statement follows.
\end{proof}

\begin{theorem}
  Let \(M=S^2\times k\C P^2\), \(k\geq 2\). Then every \(T^3\)-action on \(M\) is conjugated in \(\Diff(M)\) to a product action of an action of \(T^2\) on \(k\C P^2\) and an action of \(S^1\) on \(S^2\).
\end{theorem}
\begin{proof}
  We have
  \begin{equation*}
    H^*(M)=\mathbb{Z}[x_1,\dots,x_k,y]/(x_ix_j=0 \text{ if } i\neq j, x_i^2-x_j^2=0,y^2=0).
  \end{equation*}
Since the cohomology of \(M\) is generated in degree two, it follows from the results of \cite{1111.57019} that all intersections of characteristic manifolds of \(M\) are connected.
Moreover, the Poincar\'e duals of such intersections vanish if and only if the intersection is empty. 

Because the cohomology ring of \(M\) determines the \(f\)-vector of \(M/T\) for every \(T\)-action on \(M\) \cite[p. 736-737]{1111.57019}, \(M/T\) has \(k+4\) facets and \(2k+4\) vertices.
Let \(F_1,\dots,F_{k+4}\) be the facets of \(M/T\) and \(M_1,\dots,M_{k+4}\) the corresponding characteristic submanifolds of \(M\).
Because \(H^*(M)\rightarrow H^*(M_i)\) is surjective \cite[Lemma 2.3, p. 716]{1111.57019}, \(F_i\) has at most \(k+3\) vertices.

We first consider the case, where \(F_i\) has less than \(k+3\) vertices.
Then \(H^2(M)\rightarrow H^2(M_i)\) has non-trivial kernel.
Let \(\beta y + \sum_{i=1}^k \alpha_i x_i\) be an element of this kernel.

At first assume that \(\beta\neq 0\) and \((\alpha_1,\dots,\alpha_k)\neq (0,\dots,0)\).
Then
\begin{equation*}
  -\sum_{i=1}^k \alpha_i^2 x_i^2 = (\beta y + \sum_{i=1}^k \alpha_i x_i)(\beta y - \sum_{i=1}^k \alpha_i x_i) 
\end{equation*}
is an element of the kernel of \(H^4(M)\rightarrow H^4(M_i)\).
Therefore \(x_i^2\) is an element of this kernel.
This implies that
\begin{equation*}
  \beta y x_i = (\beta y + \sum_{j=1}^k \alpha_j x_j)x_i -\alpha_i x_i^2
\end{equation*}
is also an element of this kernel.
Therefore the map \(H^4(M)\rightarrow H^4(M_i)\) is trivial.
But this is impossible because \(H^*(M)\rightarrow H^*(M_i)\) is surjective and \(H^4(M_i)=\mathbb{Z}\).

Therefore we must have \(\beta=0\) or \((\alpha_1,\dots,\alpha_k)=(0,\dots,0)\).

At first assume that \(\beta\neq 0\) and \((\alpha_1,\dots,\alpha_k)=(0,\dots,0)\).
Because \(H^4(M)\rightarrow H^4(M_i)=\mathbb{Z}\) is surjective, we have
\begin{align*}
  \kernel (H^4(M)\rightarrow H^4(M_i)) &=  \bigoplus_{i=1}^k \mathbb{Z} yx_i,\\
  \kernel (H^2(M)\rightarrow H^2(M_i)) &= \mathbb{Z} y.
\end{align*}

Therefore \(F_i\) has \(k+2\) vertices.
Moreover the Poincar\'e-dual of \(M_i\) is given by \(\pm y\).
Therefore the intersection of two such facets is empty.

Now assume that \(\beta=0\) and \((\alpha_1,\dots,\alpha_k)\neq(0,\dots,0)\).
Then \(x_i^2\) is an element of the kernel of \(H^4(M)\rightarrow H^4(M_i)\).
Let \(\sum_{i=1}^k \gamma_i x_i \in H^2(M)\) such that its restriction to \(H^2(M_i)\) is non-trivial.
Then there are \(\delta, \epsilon_i \in \mathbb{Z}\) such that
\begin{equation*}
  (\delta y+\sum_{i=1}^k \epsilon_i x_i)(\sum_{i=1}^k\gamma_i x_i)[M_i]\neq 0.
\end{equation*}
Therefore \(\delta\neq 0\) and \(y\sum_{i=1}^k\gamma_i x_i\) is not contained in the kernel of \(H^4(M)\rightarrow H^4(M_i)\).
This implies
\begin{align*}
  \kernel (H^4(M)\rightarrow H^4(M_i)) &=  \mathbb{Z}x_i^2\oplus y \kernel (H^2(M)\rightarrow H^2(M_i)),\\
  \rank \kernel (H^2(M)\rightarrow H^2(M_i)) &= k-1.
\end{align*}
Therefore \(F_i\) has four vertices and the Poincar\'e dual of \(M_i\) is contained in \(\bigoplus_{i=1}^k \mathbb{Z}x_i\).
Therefore the intersection of three such facets is empty.

Therefore \(M/T\) has the following three types of facets:
\begin{enumerate}
\item\label{item:1} facets \(F_i\) with \(k+3\) vertices. Denote the number of facets of this type by \(d_{k+3}\).
\item \label{item:2} facets \(F_i\) with \(k+2\) vertices such that the Poincar\'e-dual of \(M_i\) is given by \(\pm y\). 
Denote the number of facets of this type by \(d_{k+2}\).
\item \label{item:3} facets \(F_i\) with \(4\) vertices such that the Poincar\'e-dual of \(M_i\) is contained in \(\bigoplus_{i=1}^k \mathbb{Z}x_i\).
Denote the number of facets of this type by \(d_4\).
\end{enumerate}

Because \(M/T\) has \(2k+4\) vertices and \(k+4\) facets it follows that
\begin{equation*}
  \left(
    \begin{matrix}
      6k+12\\
      k+4
    \end{matrix}
\right)=
 \left(
    \begin{matrix}
      4&k+2&k+3\\
      1&1&1
    \end{matrix}
\right)
 \left(
    \begin{matrix}
      d_4\\
      d_{k+2}\\
      d_{k+3}
    \end{matrix}
\right).
\end{equation*}
If \(k>3\), then the only solution in \(\N^3\) of this equation is given by
\begin{equation*}
   (d_4,d_{k+2},d_{k+3})=(k+2,2,0).
\end{equation*}

If \(k=3\), then this equation has the following two solutions
\begin{equation*}
   (d_4,d_{k+2},d_{k+3})=(k+2,2,0)
\end{equation*}
 and
 \begin{equation*}
    (d_4,d_{k+2},d_{k+3})=(k+3,0,1).
 \end{equation*}
By Theorem 8.3 of \cite[p. 738]{1111.57019}, each vertex of \(M/T\) is the intersection of three facets of \(M/T\).
Therefore, in the second case, there must be vertices of \(M/T\) which are the intersection of three facets of type \ref{item:3}.
But this is impossible.
Therefore this case does not occur.

If \(k=2\), then we have \(6=d_4+d_{k+2}\) and \(d_{k+3}=0\).
Therefore all facets of \(M/T\) have four vertices.
Hence, \(M/T\) is a cube.

Because the intersection of two facets of type \ref{item:2} and the intersection of three facets of type \ref{item:3} are empty, it follows that \(d_{k+2}=2\) and \(d_{4}=4=k+2\).

Therefore the structure of the Poincar\'e-duals of the characteristic submanifolds of \(M\) is the same as the structure of Poincar\'e-duals which is obtained by a product action.
This product action may be constructed from the \(T\)-action on \(M\) as follows.

Let \(M_1\) be the preimage of a facet of type \ref{item:2}.
Then \(M_1\) is naturally a four-dimensional torus manifold with \(b_1(M_1)=0\).
By the classification of four-dimensional torus manifolds given by Orlik and Raymond \cite{0216.20202}, \(M_1\) is diffeomorphic to \(k\C P^2\).
Therefore the product \(M_1\times S^2\) is diffeomorphic to \(M\) and Theorem~\ref{sec:class-six-dimens-2} implies that the \(T\)-action on \(M\) is conjugated to the product action on \(M_1\times S^2\).
\end{proof}

The previous theorem together with Lemma~\ref{sec:tori-diff-group-3} implies that the number of conjugacy classes of three-dimensional tori in \(\Diff(M)\), \(M=S^2\times k\C P^2\), \(k\geq 2\), is equal to the number of conjugacy classes of two dimensional tori in \(\Diff(k\C P^2)\).

\bibliography{exotic}{}

\providecommand{\bysame}{\leavevmode\hbox to3em{\hrulefill}\thinspace}
\providecommand{\MR}{\relax\ifhmode\unskip\space\fi MR }
\providecommand{\MRhref}[2]{%
  \href{http://www.ams.org/mathscinet-getitem?mr=#1}{#2}
}
\providecommand{\href}[2]{#2}
\begin{thebibliography}{10}

\bibitem{0246.57017}
G.~E. Bredon, \emph{{Introduction to compact transformation groups.}}, {Pure
  and Applied Mathematics, 46. New York-London: Academic Press. XIII}, 1972
  (English).

\bibitem{choi_masuda_suh_pre}
S.~Choi, M.~Masuda, and D.~Y. Suh, \emph{Rigidity problems in toric toplogy, a
  survey}, to appear in Proc. Steklov Inst. Math. (2011).

\bibitem{1212.57009}
D.~J. Crowley and P.~D. Zvengrowski, \emph{{On the non-invariance of span and
  immersion co-dimension for manifolds.}}, Arch. Math., Brno \textbf{44}
  (2008), no.~5, 353--365 (English).

\bibitem{0403.57002}
M.~Davis, \emph{{Smooth G-manifolds as collections of fiber bundles.}}, Pac. J.
  Math. \textbf{77} (1978), 315--363 (English).

\bibitem{0733.52006}
M.~W. Davis and T.~Januszkiewicz, \emph{{Convex polytopes, Coxeter orbifolds
  and torus actions.}}, Duke Math. J. \textbf{62} (1991), no.~2, 417--451
  (English).

\bibitem{0214.50303}
R.D. Edwards and R.C. Kirby, \emph{{Deformations of spaces of imbeddings.}},
  Ann. Math. \textbf{93} (1971), 63--88 (English).

\bibitem{0148.43103}
D.B.A. Epstein, \emph{{The degree of a map.}}, Proc. Lond. Math. Soc., III.
  Ser. \textbf{16} (1966), 369--383 (English).

\bibitem{0139.16903}
W.~Hsiang, \emph{{On the classification of differentiable SO$(n)$ actions on
  simply connected $\pi$-manifolds.}}, Am. J. Math. \textbf{88} (1966),
  137--153 (English).

\bibitem{0153.53703}
K.~J\"{a}nich, \emph{{Differenzierbare Mannigfaltigkeiten mit Rand als
  Orbitr\"aume differenzierbarer G-Mannigfaltigkeiten ohne Rand.}}, Topology
  \textbf{5} (1966), 301--320 (German).

\bibitem{0264.57004}
F.E.A. Johnson, \emph{{Lefschetz duality and topological tubular
  neighbourhoods.}}, Trans. Am. Math. Soc. \textbf{172} (1972), 95--110
  (English).

\bibitem{0176.22004}
R.C. Kirby, \emph{{Stable homeomorphisms and the annulus conjecture.}}, Ann.
  Math. \textbf{89} (1969), 575--582 (English).

\bibitem{0767.57001}
A.~A. Kosinski, \emph{{Differential manifolds.}}, {Pure and Applied
  Mathematics, 138. Boston, MA: Academic Press. xvi, 248 p. }, 1993 (English).

\bibitem{1111.57019}
M.~Masuda and T.~Panov, \emph{{On the cohomology of torus manifolds.}}, Osaka
  J. Math. \textbf{43} (2006), no.~3, 711--746 (English).

\bibitem{1160.57032}
M.~Masuda and D.~Y. Suh, \emph{{Classification problems of toric manifolds via
  topology.}}, {Harada, Megumi (ed.) et al., Toric topology. International
  conference, Osaka, Japan, May 28--June 3, 2006. Providence, RI: American
  Mathematical Society (AMS). Contemporary Mathematics 460, 273-286 (2008).},
  2008.

\bibitem{0486.57016}
P.~Melvin, \emph{{Tori in the diffeomorphism groups of simply-connected
  4-manifolds.}}, Math. Proc. Camb. Philos. Soc. \textbf{91} (1982), 305--314
  (English).

\bibitem{0216.20202}
P.~Orlik and F.~Raymond, \emph{{Actions of the torus on 4-manifolds. I.}},
  Trans. Am. Math. Soc. \textbf{152} (1970), 531--559 (English).

\bibitem{0287.57017}
\bysame, \emph{{Actions of the torus on 4-manifolds. II.}}, Topology
  \textbf{13} (1974), 89--112 (English).

\bibitem{0449.57009}
G.~W. Schwarz, \emph{{Lifting smooth homotopies of orbit spaces.}}, Publ.
  Math., Inst. Hautes \'Etud. Sci. \textbf{51} (1980), 37--135 (English).

\bibitem{0311.57001}
I.A. Volodin, V.E. Kuznetsov, and A.T. Fomenko, \emph{{The problem of
  discriminating algorithmically the standard three- dimensional sphere.}},
  Russ. Math. Surv. \textbf{29} (1974), no.~5, 71--172 (English).

\bibitem{wiemeler-remarks}
M.~Wiemeler, \emph{Remarks on the classification of quasitoric manifolds up to
  equivariant homeomorphism}, Arch. Math. \textbf{98} (2012), no.~1, 71--85.

\bibitem{0823.52002}
G.~M. Ziegler, \emph{{Lectures on polytopes.}}, {Berlin: Springer-Verlag}, 1995
  (English).

\end{thebibliography}
\bibliographystyle{amsplain}
\end{document}